\documentclass[11pt, a4paper]{amsart}
\usepackage[latin1]{inputenc}
\usepackage{amsthm}
\usepackage{amssymb}
\usepackage{tikz}
\usepackage{amsmath}
\usepackage{cancel}
\usepackage{a4wide}
\usepackage{verbatim}
\usepackage[normalem]{ulem}
\usepackage{hyperref}
\usepackage[capitalize]{cleveref}

\usepackage[foot]{amsaddr}
\numberwithin{equation}{section} 
\newtheorem{thm}{Theorem}[section]

\theoremstyle{definition}
\newtheorem{defn}[thm]{Definition}
\theoremstyle{plain}
\newtheorem{prop}[thm]{Proposition}

\newcommand{\func}[4]{\begin{cases}
		#1 \longrightarrow #2  \\
		#3 \mapsto #4
	\end{cases}
}

\newtheorem{lem}[thm]{Lemma}
\newtheorem{cor}[thm]{Corollary}
\theoremstyle{remark}
\newtheorem{rem}[thm]{Remark}

\newtheorem{question}[thm]{Question}

\usepackage{graphicx}
\usepackage{amscd}
\usepackage[all,knot,poly]{xy}
\usepackage[english]{babel}
\usepackage{xcolor}

\usepackage{mathrsfs}
\usepackage{hyperref}
\newcommand{\F}{\mathbb{F}}

\newcommand{\C}{\mathbb{C}}

\newcommand{\Z}{\mathbb{Z}}

\newcommand{\Rad}{{\rm Rad}}
\newcommand{\s}{\text{ss}}

\newcommand{\one}{\mathbb{\mathbf{1}}}

\newcommand{\Spec}{\text{Spec}}

\title[On prime Cayley graphs]{On prime Cayley graphs$^{\ast \ast}$}
\author{Maria Chudnovsky$^{\ast \amalg}$}
\author{Michal Cizek$^{\dagger \ddagger}$}
\author{Logan Crew$^{\mathsection \mathparagraph}$}
\author{J\'an Min\'a\v{c}$^{\dagger \star }$}
\author{Tung T. Nguyen $^\dagger$}
\author{Sophie Spirkl$^{\mathsection \parallel}$}
\author{Nguy$\tilde{\text{\^{e}}}$n Duy T\^{a}n$^{\ddagger \&}$}

\address{$^{\ast}$Princeton University, Princeton, NJ, USA}
\address{$^{\dagger}$ Department of Mathematics, Western University, London, Ontario, Canada N6A 5B7}
\address{$^{\mathsection}$ Department of Combinatorics and Optimization, University of Waterloo, Waterloo, Ontario, Canada}
\address{$^{\ddagger}$ Hanoi University of Science and Technology, Hanoi, Vietnam}

\address{$^{\ast \ast}$ This project was initiated during a visit to Western Academy and was supported by Western Academy for Advanced Research.}
\address{$^{\amalg}$ Supported by NSF-EPSRC Grant DMS-2120644 and by AFOSR grant FA9550-22-1-0083.}
\address{$^{\parallel}$ We acknowledge the support of the Natural Sciences and Engineering Research Council of Canada (NSERC), [funding reference number RGPIN-2020-03912].
Cette recherche a \'et\'e financ\'ee par le Conseil de recherches en sciences naturelles et en g\'enie du Canada (CRSNG), [num\'ero de r\'ef\'erence RGPIN-2020-03912]. This project was funded in part by the Government of Ontario.}
\address{$^{\mathparagraph}$ We acknowledge the support of the Natural Sciences and Engineering Research Council of Canada (NSERC), [funding reference number RGPIN-2022-03093].
Cette recherche a \'et\'e financ\'ee par le Conseil de recherches en sciences naturelles et en g\'enie du Canada (CRSNG), [num\'ero de r\'ef\'erence RGPIN-RGPIN-2022-03093]. }
\address{$^\star$ Supported by the Natural Sciences and Engineering Research Council of Canada ( NSERC) [funding reference number R037A01] and  Western Academy for Advanced Research.}
\address{$^\ddagger$ Supported by   Western Academy for Advanced Research.}
\address{$^\&$ Funded by Vingroup Joint Stock Company and supported by Vingroup Innovation Foundation (VinIF) under the project code
VINIF.2021.DA00030.}

\usepackage[parfill]{parskip}
\date{}

\begin{document}

\maketitle
\begin{abstract}

The decomposition of complex networks into smaller, interconnected components is a central challenge in network theory with a wide range of potential applications. In this paper, we utilize tools from group theory and ring theory to study this problem when the network is a Cayley graph. In particular, we answer the following question: Which Cayley graphs are prime?
\end{abstract}

\newpage 
\tableofcontents

\section{Introduction}

Network theory plays a fundamental role in the study and understanding of collective behaviors of several systems \cite{boccara2010modeling,strogatz2001exploring, watts1998collective}. For example, it has found applications in neuroscience, ecology, neural dynamics, biology, and many physical systems \cite{banerjee2016chimera, bassett2018nature, dorfler2013synchronization, kinney2005modeling, motter2013spontaneous, silk2018can}.  Typically, we often consider these systems as a single network. However, many real-world systems are multi-layered; namely, the system can have two levels of connections: on the first level, there are connections within each sub-network; on the higher level, there are connections between the sub-networks. While the addition of an extra level of connections provides a more realistic and richer modeling option, it can also bring new challenges in analyzing the dynamics of these systems. The quest to understand the dynamics of phase oscillators on these multilayer networks has sparked a lot of interest due to its wide range of applications (see \cite{boccaletti2014structure, gao2011robustness,gao2012networks, jain2023composed, kenett2015networks, kivela2014multilayer,menara2019stability, schaub2016graph,tiberi2017synchronization}). 
In \cite{nguyen2022broadcasting}, we introduce a new method to study this problem using a reduced and much smaller representation system.
More precisely, we discover a process that allows us to broadcast solutions from the reduced system to the original multilayer network. Amongst its various applications, our new approach offers a simple way to find equilibrium points and analyze their stability on a multilayer network. It is worth mentioning that our approach is based on a concept in graph theory known as the ``joined union'' of graphs which we now recall. Let $G$ be a graph with $d$ vertices $\{v_1, v_2, \ldots, v_d\}$. Let $G_1, G_2, \ldots, G_d$ be graphs representing layers in a given multilayer network. The joined union $G[G_1, G_2, \ldots, G_d]$ is obtained from the union of $G_1, \ldots, G_d$ by joining with an edge each pair of a vertex from $G_i$ and a vertex from $G_j$ whenever $v_i$ and $v_j$ are adjacent in $G$ (see \cite{joined_union2, CM1, joined_union} for some further discussions). In this framework, $G$ represents precisely the connections between layers.

Due to its potential wide range of applications as explained above, the following question seems natural.

\begin{question} \label{que: fundamental}
    Given a graph $\Gamma$, can we decompose it as the joined union of smaller graphs? 
\end{question}

\subsection{Homogeneous sets}

It turns out that the above question has been extensively studied by graph theorists but from a different perspective. Specifically, it is studied through the theory of modular decomposition using homogeneous sets which we now recall (see, for example, \cite{brandstadt1999graph, gallai1967transitiv}).  A \emph{homogeneous set} or \emph{module} in a graph $\Gamma$ is a set $X$ of vertices of $\Gamma$ such that every vertex in $V(\Gamma) \setminus X$ is adjacent to either all or none of the vertices in $X$. Thus, vertices in $V(\Gamma) \setminus X$ do not ``distinguish'' vertices in $X$. This makes homogeneous sets a natural generalization of connected components (which are, in particular, homogeneous sets $X$ such that every vertex outside $X$ is non-adjacent to every vertex in $X$). Note that $V(\Gamma)$ as well as all vertex sets of size at most one are homogeneous sets; we refer to homogeneous sets $X$ with $2 \leq X < |V(\Gamma)|$ as \emph{non-trivial homogeneous sets}. We remark that the theory of homogeneous sets is closely related to \cref{que: fundamental} in the sense that for a graph join $\Gamma = G[G_1, G_2, \ldots, G_d]$, each copy of $G_i$ is a homogeneous set of $\Gamma$. Conversely, a graph contains a non-trivial homogeneous set if and only if it can be composed as $G[G_1, G_2, \ldots, G_d]$ where at least one of the $G_i$ has more than one vertex.

Homogeneous sets are a well-studied structure in graph theory, and can be found in polynomial time \cite{cournier1993quelques, muller1989incremental, spinrad1992p4}. Their inverse operation is \emph{substitution}: a vertex $v$ in a graph $\Gamma$ is replaced by a graph $H$ with vertex set $X$, and for every vertex $u$ in $V(\Gamma) \setminus \{v\}$, we add all edges from $u$ to $X$ if and only if $uv \in E(\Gamma)$; so $X$ is a homogeneous set in the resulting graph. 

The class of perfect graphs plays an important role in graph theory \cite{berge1961deren, chudnovsky2006strong}  and optimization \cite{grotschel1984polynomial}. This class is closed under substitution \cite{lovasz1972normal}, meaning that if $\Gamma$ and $H$ are perfect, then the graph obtained by substituting $H$ for a vertex of $\Gamma$ is again a perfect graph.  In addition, and in part due to this connection with perfect graphs, homogeneous sets are useful in more than one way  \cite{alon2001ramsey, chudnovsky2014erdos, chudnovsky2008erdHos} in efforts to prove the Erd\H{o}s-Hajnal conjecture \cite{erdos1989ramsey}, one of the most notable open questions in structural and extremal graph theory. Substitution also interacts well with twin-width \cite{bonnet2021twin} (that is, substituting a graph $H$ for a vertex of a graph $\Gamma$ results in a graph whose twin-width is the maximum of the twin-widths of $\Gamma$ and $H$). Twin-width is a parameter describing the complexity of a structure;  it is used not only for graphs, but also for groups (via Cayley graphs) \cite{bonnet2022twin} and permutations as well as other binary structures \cite{bonnet2021twin2}.

In this paper, we will concentrate on the case where $\Gamma$ is a Cayley graph (see \cite{cayley1878desiderata, krebs2011expander}). 
We consider the following question: Which properties of $G$ and $S$ lead to the existence of a non-trivial homogeneous set in $\text{Cay}(G,S)$? Which Cayley graphs are \emph{prime}, that is, they do not admit a non-trivial homogeneous set? Because Cayley graphs are highly symmetric, the criteria for the presence of a homogeneous set turn out to be similar to those given by Barber \cite{barber2021recognizing} for $\text{Cay}(G, S)$ to admit a decomposition as a wreath product of smaller graphs.

\subsection{Outline}

This paper is meant to be as self-contained as possible so that it is accessible to readers possessing different backgrounds. The structure of our article is as follows. In \cref{sec:def}, we recall some standard concepts in graph theory. In \cref{sec:homogeneous_set_cayley},  we will then establish a connection between the existence of non-trivial homogeneous sets and decompositions as wreath products for vertex-transitive graphs. This leads to a criterion for when Cayley graphs are prime. Finally, we will briefly look into the case of directed graphs and show analogous results for them as well.  \cref{sec:ring_cayley} studies Cayley graphs associated with a ring $R$. Here, the interplay between the additive and multiplicative structure of $R$ plays a fundamental role in our investigation. 

Amongst the various results that we prove in this section, we give a complete classification of prime unitary graphs.

\section{Definitions and notation}
\label{sec:def}
Throughout this article, unless specified otherwise, all graphs are undirected, simple, and finite. One can define an undirected simple finite graph as follows:

\begin{defn}[Graph]
  A simple undirected finite graph $\Gamma$ is a pair $(V(\Gamma) , E(\Gamma ))$, where $V(\Gamma )$ is a finite set and $E(\Gamma )$ is an anti-reflexive symmetric relation on $\Gamma$. In other words, $E(\Gamma)$ is a set of pairs $(x,y) \in V(\Gamma )^{2}$, such that $(x,x) \notin E(\Gamma )$ for every $x \in V(\Gamma )$, and if $(x,y) \in E(\Gamma )$, then $(y,x) \in E(\Gamma )$.

  The set $V(\Gamma )$ is called the set of vertices of $\Gamma$, and $E(\Gamma )$ is the set of edges. For a pair $x, y$ of vertices, we say that $xy$ is an edge (or that $x$ and $y$ are adjacent) if $(x,y) \in E(\Gamma )$. 
\end{defn}

\begin{defn}[Graph morphism] \label{def:graph_morphism}
 Let $\Gamma$ and $\Delta$ be two graphs. We define a graph morphism between $\Gamma$ and $\Delta$ to be a map from $V(\Gamma )$ to $V(\Delta)$ $f$, such that if $u,v \in V(\Gamma )$ are two adjacent vertices in $\Gamma$, then $f(u),f(v)$ are two adjacent vertices in $\Delta$.
\end{defn}

The class of graphs we are interested in particular is the class of Cayley graphs. We define a Cayley graph as follows:

\begin{defn}[Cayley graph]
  Let $G$ be a group and $S \subseteq G\setminus \{1\}$ a set such that if $s \in S$, then $s^{-1} \in S$.
  We define the Cayley graph with generators in $S$ to be the graph $\Gamma = \text{Cay}(G,S)$, with $V(\Gamma) = G$ and $E(\Gamma) = \{(x,y) \in G^{2} | x^{-1}y \in S\}$.
\end{defn}

We will now introduce the definition of a group action on graphs.

\begin{defn}[Group acting on a graph]
  Let $G$ be a group and $\Gamma$ a graph. We say that $G$ acts on $\Gamma$ if there is an action of $G$ on  $V(\Gamma )$, such that if $xy$ is an edge in $\Gamma$, then $[g\cdot x] [g\cdot y]$ is an edge in $\Gamma$.

  We call a graph $\Gamma$ vertex-transitive if there is a group $G$ acting on $\Gamma$ such that the action of $G$ on $V(\Gamma )$ is transitive.
\end{defn}

Now we will define the wreath product of graphs.

\begin{defn}[Wreath product]
  Let $\Gamma ,\Delta$ be two graphs. We define the wreath product (also known as the lexicographic product) of $\Gamma$ and $\Delta$ as the graph $\Gamma \cdot \Delta$, whose set of vertices is $V(\Gamma ) \times V(\Delta )$ and such that $(x,y)(x',y')$ is an edge in $\Gamma \cdot \Delta$ if either $xx'$ is an edge in $\Gamma$, or $x=x'$ and $yy'$ is an edge in $\Delta$.
\end{defn}

\begin{rem} \label{rem:wreath_join}
    The wreath product is a special case of the joined union. Specifically, the wreath product of $\Gamma$ and $\Delta$ is exactly the joined union $\Gamma[\Delta, \cdots, \Delta]$.
\end{rem}

\begin{defn}[Connectedness]
  Let $\Gamma$ be a graph. We call connected components of $\Gamma$, the equivalence classes on $V(\Gamma )$ for the following relation:
  Two vertices $x$ and $y$ are equivalent if and only if there exists a finite sequence of vertices $x_{1}, \cdots , x_{n}$ such that for all $i$ between $1$ and $n-1$, $x_{i}$ and $x_{i+1}$ are either adjacent or equal, with $x_{1}=x$ and $x_{n}=y$.
  We say that the graph $\Gamma$ is connected if it has exactly one connected component.
\end{defn}

\begin{defn}[Induced subgraph]
  Let $\Gamma$ be a graph and let $H \subseteq V(\Gamma )$ be a non-empty subset. We define $\Gamma[H]$ to be the subgraph of $\Gamma$ induced by $H$, that is: $V(\Gamma[H]) = H$ and $E(\Gamma[H]) = \{ (x, y) \in E(\Gamma) | x, y \in H\}$.
\end{defn}

\begin{defn}[Graph complement]
    Given a graph $\Gamma$, the complement $\Gamma^*$ of $\Gamma$ is the graph such that for distinct $x, y \in V(\Gamma)=V(\Gamma^*)$, we have that $xy$ is an edge in exactly one of $\Gamma$ and $\Gamma^*$.

    A graph is anti-connected if its complement is connected. 
\end{defn}

\begin{defn}[Valence]
 If $\Gamma$ is a graph and $v \in V(\Gamma )$ a vertex, we define the valence (or degree ) of $v$, as the number of vertices adjacent to $v$.
\end{defn}

Finally, we recall the definition of homogeneous sets.

\begin{defn}[Homogeneous sets]
  Let $\Gamma$ be a graph and $H$ a non-empty subset of $V(\Gamma )$. We say that $H$ is homogeneous (in $\Gamma$) if, for every $x\notin H$, either $x$ is adjacent to all elements of $H$ or to none. A homogeneous set $H$ is non-trivial if $2 \leq |H| < |V(\Gamma)|$. 
\end{defn}

We will show how the notion of a homogeneous set can help to detect whether a vertex-transitive graph can be written as a non-trivial wreath product and we will apply it more specifically to Cayley graphs, as well as more precisely to Cayley graphs over rings.

\section{Homogeneous sets and Cayley graphs}
\label{sec:homogeneous_set_cayley}
\subsection{Homogeneous sets and the wreath product}\leavevmode

First note that any graph $\Gamma$ is isomorphic to a wreath product of the form $* \cdot \Gamma$ and $\Gamma \cdot *$, where $*$ is a graph with a single vertex. Therefore, we call such a product \emph{trivial}. The question that we will examine is when a given graph is isomorphic to a non-trivial wreath product. Note that if $\Gamma \cdot \Delta$ is a non-trivial wreath product, then it has non-trivial homogeneous sets of the form $\{x\} \times V(\Delta )$, where $x$ is a vertex in $V(\Gamma )$. 

From this simple observation, we have a necessary condition: that every graph that is a non-trivial wreath product has a non-trivial homogeneous set.
We call a graph with no non-trivial homogeneous sets \emph{prime}. 

The presence non-trivial homogeneous set does not imply that a graph is a non-trivial wreath product: Consider the graph with vertices $\{1,2,3\}$ such that $1$ is adjacent to $2$ and $3$ and $2,3$ are non-adjacent:

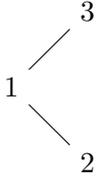
\begin{tikzpicture}
  \node(A) at (0,0){$1$};
  \node(B) at (1,-1){$2$};
  \node(C) at (1,1){$3$};
  \draw[-] (A) to (B);
  \draw[-] (A) to (C);
\end{tikzpicture}

Note that the set $\{2,3\}$ is homogeneous and non-trivial and so the graph is not prime, but it is not a non-trivial wreath product, because it has $3$ vertices and $3$ is a prime number.

We can however show that for vertex-transitive graphs, except for two extreme cases, namely \emph{complete} graphs (all vertices are adjacent) and \emph{cocomplete} graphs (no vertices are adjacent), the converse is true.

For that, we will need to establish a simple lemma (we believe this is well-known, but include a proof for completeness):
\begin{lem} \label{maximum lemma}
 Let $\Gamma$ be a graph and let $H,H'$ be two non-disjoint homogeneous sets. Then $H \cup H'$ is a homogeneous set.
\end{lem}

\begin{proof}
  Take $x \in H \cap H'$. Take $y \notin H \cup H'$, such that $y$ is adjacent to some element in $z \in H\cup H'$. Without loss of generality, we may assume that $z \in H$, since there is symmetry between $H$ and $H'$. Since $y$ and $z$ are adjacent, it follows $y$ is adjacent to every element in $H$ and in particular to $x$. Since $x$ is also an element of $H'$, this implies that $y$ is adjacent to every element in $H'$ and thus $H\cup H'$ is homogeneous.
\end{proof}

Now we will show a criterion for when a vertex-transitive graph is a non-trivial wreath product.

\begin{thm} \label{homogeneous product equivalence}
Let $\Gamma$ be a vertex-transitive graph that is not complete nor cocomplete. Then $\Gamma$ is a non-trivial wreath product if and only if it is not prime.
\end{thm}

\begin{proof}
  We have already observed that if $\Gamma $ is a non-trivial wreath product then it has non-trivial homogeneous sets.

  Now conversely assume that $\Gamma $ has a homogeneous set $X$.
  By assumption, $\Gamma$ is vertex-transitive and thus there exists $G$ a group acting on $\Gamma$, such that $G$ acts transitively on $V(\Gamma )$.
  We will now distinguish three cases:

  \textbf{Case 1:} Assume that $\Gamma$ is a disconnected graph. Consider then $C$ a connected component. Since by assumption, $\Gamma$ is not a graph with no edges, we can assume additionally that $|C|>1$. Note that if $C'$ is another connected component, then since $G$ acts transitively on $\Gamma$, there exists $g \in G$, such that $gC=C'$. Now choose $g_{1}, \cdots ,g_{n} \in G$, such that for every $C'$ connected component in $C$, there exists a unique $i$ such that $C'=g_{i}C$.

  Now take $A$ to be the graph whose vertices are $g_{1},\cdots ,g_{n}$ and with no edges. Let us denote by $\Delta$ the subgraph induced by $C$. Note then that $\Gamma$ is isomorphic to $A\cdot \Delta$ with the isomorphism being the map that to $(g_{i},x)$ associates $g_{i}\cdot x$.

  \textbf{Case 2:} Assume now that $\Gamma$ is not anti-connected. Consider $\Gamma^{*}$ the complement graph to $\Gamma$: that is a graph with vertices $V(\Gamma )$ and for two distinct vertices $x,y$, $xy$ is an edge in $\Gamma^{*}$ if and only if $xy$ is not an edge in $\Gamma$.

  Now if $\Gamma$ is not anti-connected, then its complement $\Gamma^{*}$ is not connected and not the graph with no edges, and so by the first case, there exist two graphs $A,H$ such that $\Gamma^*$ is isomorphic to the non-trivial wreath product $A\cdot \Gamma[H]$. Therefore $\Gamma$ is isomorphic to the non-trivial wreath product $A^{*} \cdot \Gamma^*[H]$.

  \textbf{Case 3:} Finally assume that $\Gamma$ is both connected and anti-connected. As a reminder, $X$ denotes a non-trivial homogeneous set of $\Gamma$. Choose a distinguished point $x_{0} \in X$ and pick $H$ a maximal (in the sense of inclusion) non-trivial homogeneous set of $\Gamma$ containing $x_{0}$. Such a maximal set exists since $\Gamma$ is finite and thus any non-empty set of subsets  of $V(\Gamma )$ has a maximal element. Let us now show that for every $g \in G$, either $gH \cap H=\emptyset$ or $gH=H$. Suppose that $gH \cap H \neq \emptyset$. Note that by \ref{maximum lemma} $gH \cup H$ is a homogeneous set containing $H$ so by maximality of $H$, either $gH\cup H =V(\Gamma )$, or $gH\cup H=H$. Let us show that $gH\cup H \neq V(\Gamma )$. In order to prove this, we will first show that $|H|\leq \frac{|V(\Gamma )|}{2}$. Since $\Gamma$ is connected and $H \neq V(\Gamma )$, there exists $y \notin H$ such that $y$ is adjacent to an element of $H$. The valence of $y$ then has to be at least $|H|$, since $y$ is connected to all elements of $H$.
  Now, since $\Gamma$ is also anti-connected, we see that $\Gamma^{*}$ is connected. Therefore there exists a $y' \in V(\Gamma ) \setminus H$ such that $y'$ is adjacent to some vertex in $H$ in the graph $\Gamma^{*}$. Therefore $y'$ is not adjacent to some element in $H$ in the original graph $\Gamma$. Since $H$ is homogeneous, it follows that $y'$ is not adjacent to any element in $H$ and thus the valence of $y'$ is at most $|V(\Gamma )|-|H|$. Since the graph $\Gamma$ is vertex-transitive, all of its vertices have the same valence, which we shall denote $d$. Since the valence of $y$ is greater than $|H|$, we have that $d \geq |H|$ and since the valence of $y'$ is at most $|V(\Gamma )|$, we also have $d\leq |V(\Gamma )|$. As such $$|H| \leq d \leq |V(\Gamma )-|H|$$
  Therefore:
  \begin{align*}
      |gH \cup H| & = |gH| +|H| -|gH\cap H| \\
                  & = 2|H| - |gH \cap H| \\
                  & \leq |V(\Gamma )| - |gH \cap H| \\
                  & < V(\Gamma )
  \end{align*}
  as we assume that $gH \cap H \neq \emptyset $.
  Since $H$ is the maximal homogeneous subset of $V(\Gamma )$ containing $x_{0}$ and distinct from $V(\Gamma )$, we get that $gH \cup H \subseteq H$ and thus $gH = H$, since it has the same number of elements as $H$.

  As a consequence, it follows that for all $g,g' \in G$, $gH=g'H$ or $gH \cap g'H =\emptyset$. Then, choose $g_{1}, \cdots, g_{n} \in G$, such that $g_{1} H , \cdots ,g_{n}H$ are all distinct and $g_{1} H \cup \cdots \cup g_{n}H =V(\Gamma )$, which is possible, since $G$ acts transitively on $\Gamma$. Define a graph $A$ whose vertices are $g_{1} , \cdots ,g_{n}$ and $g_{i} g_{j}$ is an edge in $A$ if $[g_{i}\cdot x_{0}] \ [g_{j}\cdot x_{0}] $ is an edge in $\Gamma$ (with $x_{0}$ the distinguished element in $H$). Now let us show that $\Gamma$ is isomorphic to $A\cdot \Gamma[H]$.

  Consider the map $$\Phi =\func{V(A) \cdot H}{\Gamma}{(g_{i},h)}{g_{i} \cdot h}$$ and let us show first that this map is a morphism of graphs. Suppose that $(g_{i},h)$ and $(g_{j},h')$ are adjacent. Then first assume that $g_{i}$ and $g_{j}$ are adjacent. Then we get that $g_{i}x_{0}$ and $g_{j}x_{0}$ are adjacent. Since $g_{j}H$ is homogeneous, this proves that $g_{i}x_{0}$ is adjacent to $g_{j}\cdot h'$. Since $g_{j}\cdot h'$ is adjacent to at least one element in $g_{i}H$ and $g_{i}H$ is homogeneous, it follows that $g_{j} \cdot h'$ is adjacent to $g_{i}\cdot h \in g_{i}H$. Consequently $\Phi (g_{i},h)$ and $\Phi (g_{j},h')$ are adjacent. Now instead assume that $g_{i}=g_{j}$ and $h,h'$ are adjacent. Then $g_{i}h$ and $g_{i}h'=g_{j}h'$ are adjacent and so $\Phi (g_{i},h)$ and $\Phi (g_{j},h')$ are adjacent.

  Now we shall prove that $\Phi$ is a bijection. To prove the injectivity, assume that $g_{i}h=g_{j}h'$. Then $g_{i}h \in g_{i}H \cap g_{j} H$, proving that $g_{i} H \cap g_{j}H \neq \emptyset$ and thus by construction $g_{i}=g_{j}$. And so $g_{i}h=g_{i}h'$ and therefore we also conclude that $h=h'$.

  Now let us prove the surjectivity of $\Phi$. Take $y \in V(\Gamma )$. Then, by construction, there exists $i$ such that $y \in g_{i}H$ and so $y \in {\rm im}(\Phi )$.

  Finally, let us prove that the inverse of $\Phi$ is also a morphism. Assume that $yy'$ is an edge in $\Gamma$. Pick $h,h'$ and $i,j$, such that $g_{i}h=y$ and $g_{j}h'=y'$. First, if we assume that $g_{i}=g_{j}$, then $hh'$ is an edge. Now if we assume that $g_{i} \neq g_{j}$, then since $g_{i}h$ is adjacent to $g_{j}h'$, and since $g_{i}H$ and $g_{j}H$ are both homogeneous, we get that $g_{i}x_{0}$ is adjacent to $g_{j}x_{0}$ and so $g_{i}g_{j}$ is an edge in $A$.

  Note finally that $A$ has more than one element, because $H \neq V(\Gamma )$ and thus we need at least two $g_{i}$'s to cover $V(\Gamma )$. Observe also that $H$ contains at least two elements as $X \subseteq H$ and $X$ is a non-trivial homogeneous set.

  We conclude that $\Gamma$ is a non-trivial wreath product.
\end{proof}
\begin{rem}

Note that a similar result has been established in the thesis of Rachel V. Barber \cite{barber2021recognizing} (Lemma 1, page 12) for digraphs.
The lemma states that a vertex-transitive digraph is a wreath product if and only if it can be decomposed into a ``block system'' with the property that if $B,B'$ are two blocks, then either for all vertices $b\in B$ and $b'\in B'$, $(b,b')$ is an edge, or all vertices $b \in B,b'\in B'$ are not connected by an edge. What we show here is that for undirected graphs, which are neither complete nor cocomplete, such a block system exists provided that there is a homogeneous set. For graphs that are not connected (resp.\ anti-connected), such block system consists of the connected (resp.\ anti-connected) components; and for graphs that are both connected and anti-connected, such blocks are maximal homogeneous sets.

In light of this theorem, provided we disregard the complete graphs and graphs with no edges, the existence of block systems and the existence of homogeneous sets are equivalent, and thus in the parts that follow, we will simply focus on homogeneous sets.
\end{rem}

\subsection{Homogeneous sets in Cayley graphs} \leavevmode

In this subsection, we will prove the following theorem:

\begin{thm} \label{thm:subgroup_homogeneous}
  Let ${\rm Cay}(G,S )$ be a Cayley graph such that $S \neq G\setminus \{1\}$ and $S\neq \emptyset$. Then ${\rm Cay}(G,S)$ has a non-trivial homogeneous set if and only if there exists a non-trivial subgroup $H$ of $G$ such that if $c \in S \setminus H$, then the right coset $Hc$ is included in $S$.

  Furthermore, if ${\rm Cay}(G,S)$ is both connected and anti-connected, we may choose $H$ to be the maximal homogeneous set containing $1$; and this $H$ is a subgroup of $G$. 
\end{thm}

Before proving this theorem, let us acknowledge that a very similar result already exists in the thesis of Rachel V. Barber \cite{barber2021recognizing} in the form of Corollary 1 on page 16, stating that every Cayley digraph $\text{Cay}(G,S)$ is a non-trivial wreath product of two vertex-transitive graphs if and only if there exists a non-trivial subgroup $H$ of $G$ such that $G\setminus S$ is a union of double cosets of $H$.

In light of the equivalence between the existence of non-trivial wreath products and the existence of non-trivial homogeneous sets established by Theorem  \ref{homogeneous product equivalence}, the similarity is to be expected. Corollary 1 of Barber requires $S$ to be a union of double cosets, while in our case we work with right cosets. It is important to observe that if $\Gamma$ is undirected, then in fact the two criteria are the same. Indeed, every double coset is a union of right cosets. Now conversely, if we assume that $S$ has the property that if $g \in S\setminus H$, then $Hg\subseteq S$, then since the graph is undirected, $S$ is stable under taking inverses. Thus we start by taking $g \in S \setminus H$ and then note that $g^{-1}H \subseteq S$ and so by taking inverses we get that $gH \subseteq S$ and so $HgH \subseteq S$. We will later see how the point of view of homogeneous sets can be used in the case of the directed graphs as well.

Now let us prove the theorem.
\begin{proof}
  First, suppose that $\text{Cay}(G,S)$ is not connected. Set $H=\left < S \right >$, the subgroup of $G$ generated by $S$.
  Then using our hypothesis that $S \neq \emptyset$, we see that, the induced subgraph $\text{Cay}(G,S)[H]$ of $\text{Cay(G,S)}$, is the connected component of $Cay(G,S)$ containing $1$. Since $S\subseteq H$, we see that $S\setminus H = \emptyset$. Therefore $S\setminus H$ is an empty union of right cosets. Finally we see that since $\text{Cay}(G,S)$ is not connected, $H\neq G$.

   Assume now that $\text{Cay}(G,S)$ is not anti-connected. Set $H=\left < G\setminus S \right >$ be the subgroup of $G$ generated by $G\setminus S$. Then the induced subgraph $Cay(G,S)^{*}[H]$ of complement $Cay(G,S)^{*}$ of $Cay(G,S)$ is the connected component of $Cay(G,S)^{*}$ containing $1$. Since by assumption, $Cay(G,S)$ is not a complete graph, $H \neq \{1\}$. By assumption, $Cay(G,S)^{*}$ is not connected and thus $H\neq G$. From our definition of $H$, we see that $G\setminus S \subseteq H$. For each $g \in S \setminus H$, we have $Hg \cap H = \emptyset$ and therefore $Hg\subseteq S$.

  Finally assume that $\text{Cay}(G,S)$ is both connected and anti-connected. We take $H$ to be the maximal homogeneous set containing $1$. Since $G$ acts transitively on $\text{Cay}(G,S)$ and it is both connected and anti-connected, by lemma \ref{homogeneous product equivalence}, we see that if $gH \cap H \neq \emptyset$, then $gH=H$. This property implies that $H$ is a subgroup. Indeed consider $g,h \in H$. Then $1 \in g^{-1}H \cap H$ and we see that $g^{-1}H=H$. Thus $g^{-1}h \in H$ proving that $H$ is indeed a subgroup. Take $s \in S\setminus H$. Let us show that $Hs \subseteq S$. We have that $s$ is adjacent to $1$, but is not in $H$. Since $H$ is homogeneous, this implies that $s$ is adjacent to every element in $H$ and therefore $\forall h \in H$, $h^{-1}s \in S$, which implies that $Hs \subseteq S$, as expected.
\end{proof}

While the focus of this paper is on undirected graphs, we will show in the section that follows how the technique of homogeneous sets can be used for directed graphs as well.

Finally, we can make the following observation concerning the homogeneous sets in Cayley graphs:

\begin{thm}
 Let ${\rm Cay}(G,S)$ be a Cayley graph and let $H \subseteq G$ be a homogeneous set of $G$ containing $1$. Then the subgroup  $\langle H \rangle$ generated by $H$ is also a homogeneous set.
\end{thm}
\begin{proof}

  Consider the maximal homogeneous set $M$ of $\Gamma$ such that $M$ contains $H$ and $M$ is contained in $\langle H \rangle $. Let us prove that $M=\langle H \rangle$. We can prove by induction on $n$ that $\forall g_{1}, \cdots, g_{n} \in H$, $g_{1} \cdots g_{n} \in M$. To start the induction, observe that if $n=0$ we have an empty product equal to $1$ and by assumption $1 \in H \subseteq M$. Now suppose that the statement is true for some $n \in \mathbb{N}$. Consider the product $g_{1} \cdots g_{n+1}$. Now $g_{1}M$ is homogeneous and $g_{1} \in M \cap g_{1}M$, since $M$ contains $1$ and so by \cref{maximum lemma}, it follows that $M \cup g_{1}M$ is a homogeneous set containing $H$ and contained in $\langle H \rangle$ and thus by maximality of $M$, $g_{1}M \cup M \subseteq M$. By induction we have that $g_{2} \cdots g_{n+1} \in M$ and so $g_{1} \cdots g_{n+1} \in M$. Since this statement is true for every $n \in \mathbb{N}$, $M$ contains all the product of the elements of $H$ and since $G$ is a finite group, this implies that $\langle H \rangle \subseteq M$, proving that $\langle H \rangle=M$ is homogeneous.
\end{proof}

\subsection{Cayley digraphs and the wreath product} \leavevmode

First note that we will work on simple digraphs. We can define a simple digraph as a set of vertices $V(\Gamma )$ together with an anti-reflexive relation $E(\Gamma )$. We say that $xy$ is an edge if $(x,y) \in E(\Gamma)$.
Now for a directed graph $\Gamma$, we define the underlying undirected graph of $\Gamma$, denoted $S(\Gamma )$, to be the undirected graph with vertices $V(\Gamma )$ and edges \[\{(x,y) \in V(\Gamma )^{2}| \ xy \text{ is an edge in } \Gamma \text{ or } yx \text{ is an edge in } \Gamma\}\]

We say that $\Gamma$ is \emph{connected} if $S(\Gamma )$ is connected, and that $\Gamma$ is \emph{anti-connected} if $S(\Gamma^{*})$ is connected, with $\Gamma^{*}$ being the \emph{complement (di)graph} of $\Gamma$ defined as a graph vertex with the vertex set $V(\Gamma)$ and edges of the form $\{(x, y)| x \neq y \text{ and } xy \text{ is not an edge in } \Gamma\}$.
For a digraph $\Gamma$ we say it is \emph{complete} (resp.\ \emph{cocomplete}) if $E(\Gamma^*)= \emptyset$ (resp.\ $E(\Gamma)= \emptyset$).

For a directed graph $\Gamma$ and a subset $H \subseteq V(\Gamma )$, we say that $H$ is \emph{in-homogeneous} if, for every $v \in V(\Gamma) \setminus H$, either $vh$ is an edge of $\Gamma$ for all $h \in H$, or $vh$ is not an edge of $\Gamma$ for all $h$ in $G$. 

Similarly, we say that $H \subseteq V(\Gamma )$ is \emph{out-homogeneous} if, for every $v \in V(\Gamma) \setminus H$, either $hv$ is an edge of $\Gamma$ for all $h \in H$, or $hv$ is not an edge of $\Gamma$ for all $h$ in $G$. 

Finally, $H$ is \emph{bihomogeneous} if it is both in- and out-homogeneous. As in the undirected case, we say that a bihomogeneous set is \emph{non-trivial} if it contains at least two vertices, and does not contain all vertices of $\Gamma$. Likewise, the notions of \emph{vertex-transitive} and \emph{(non-trivial) wreath product} are defined in analogy with the undirected case.
Note that just like for the undirected case, we have the following lemma:

\begin{lem} \label{maximum in-lemma}
  Let $\Gamma$ be a directed graph and $H,H'$ two bihomogeneous sets with non-empty intersections. Then $H \cup H'$ is bihomogeneous.
\end{lem}

Now we will prove the following theorem: 

\begin{thm}
   If $\Gamma$ is a vertex-transitive directed graph such that $\Gamma$ is neither complete nor cocomplete, then $\Gamma$ is a non-trivial wreath product if and only if $\Gamma$ has a non-trivial bihomogeneous set.
\end{thm}

\begin{proof}
  For the ``if'' case, we note that if $\Gamma$ is isomorphic to $\Delta \cdot \Delta '$, then $\Delta [\{x\}] \cdot \Delta '$ is a bihomogeneous set.

  Now, we take a group $G$ acting on $\Gamma$, with transitive action on $V(\Gamma )$.
  To prove the ``only if'' case we distinguish several cases:

  In the first case, assume that $\Gamma$ is not connected. Then, as in the directed case of \cref{homogeneous product equivalence}, we take the connected components $g_{1}C , \cdots ,g_{n}C$  and we use that there are no edges between them (regardless of direction). Again, $\Gamma$ is isomorphic to $A \cdot \Gamma [C]$ where  $A$ is the graph with vertex set $\{g_{1} , \cdots ,g_{n}\}$ and no edges. Note that $C$ contains more than $1$ element since $\Gamma ^{*}$ is not complete and thus $\Gamma$ has at least one edge.

  In the second case, we assume that $\Gamma$ is not anti-connected. Take then $C$ a connected component in $S(\Gamma^{*} )$. Note that for every $g\in G$, either $g \cdot C = C$ or $g\cdot C \cap C =\emptyset$. Moreover, $g \cdot C$ is a connected component in $S(\Gamma^{*} )$. To prove this, we first show that if we take $x,y \in g\cdot C$, then there exists a path in $S(\Gamma^{*} )$ from $x$ to $y$. Note that both $g^{-1} \cdot x $ and $g^{-1} \cdot y$ are in $C$ and so there exist $x_{1} , \cdots ,x_{n}$, such that $x_{1}=x,x_{n}=y$ and $x_{i}x_{i+1} \in E(S(\Gamma^{*}))$. Now take any $i$ between $1$ and $n-1$. We have that $x_{i}x_{i+1} \in E(S(\Gamma^{*} ))$ and  so $x_{i}x_{i+1} \notin E(\Gamma )$ or $x_{i+1}x_{i}\notin E(\Gamma)$. As such, $[g\cdot x_{i}][g\cdot x_{i+1}] \in E(\Gamma^{*})$ or $[g\cdot x_{i+1}][g\cdot x_{i}] \in E(\Gamma^{*})$ and either way, $[g\cdot x_{i}][g\cdot x_{i+1}] \in S(E(\Gamma^{*}))$, proving that $g\cdot C$ is connected in $S(\Gamma^{*})$, since $g\cdot x_{1} ,\cdots g\cdot x_{n}$ is a path from $x$ to $y$ in $S(\Gamma^{*})$. Now assume that $y$ is adjacent to $x \in g\cdot C$ in $S(\Gamma ^{*})$. Let us show that $y \in g\cdot C$. Since $y$ is adjacent to some $x$, it follows that $xy \in E(S(\Gamma^{*}) )$, so $xy \in E(\Gamma^{*})$ or $yx \in E(\Gamma^{*})$. As such, $[g^{-1}\cdot x][g^{-1}\cdot y] \in E(S(\Gamma^{*} ))$ and since $C$ is a connected component of $S(\Gamma^{*})$ and $x \in C$, we have that $g^{-1} \cdot y \in C$ and therefore $y \in g\cdot C$.

  Now we shall prove that if $x \in C$ and $y \notin C$, then $xy \in E(\Gamma )$ and $yx \in E(\Gamma )$. Indeed, if, say, $xy \notin E(\Gamma )$, then $xy \in E(\Gamma^{*}) \subseteq E(S(\Gamma^{*}))$ contradicting that $C$ is a connected component. By the same reasoning, we have that $yx \in E(\Gamma )$. Finally, if we pick $g_{1} ,\cdots ,g_{n} \in G$ such that $g_{1} C, \cdots ,g_{n}C$ form a partition on $V(\Gamma )$, then if $A$ denotes the complete graph on $\{g_{1}, \cdots ,g_{n}\}$, then $\Gamma$ is isomorphic to the product $A\cdot \Gamma [C]$.

  Finally, we assume $\Gamma$ is both connected and anti-connected. We take $x_{0} \in V(\Gamma )$ and let $X$ be a non-trivial bihomogeneous set containing $x_{0}$. Choose a maximal bihomogeneous set  $H$ containing $X$; let us show that for every $g\in G$, either $g\cdot H =H$ or $g\cdot H \cap H=\emptyset$. Assume that $g\cdot H \cap H \neq \emptyset$. We will then show that $g\cdot H =H$. To prove this, just like in the undirected case of \ref{homogeneous product equivalence}, we will show that $|H| \leq \frac{|G|}{2}$. Since $\Gamma$ is connected, there exists $x \notin H$ and $h_{0} \in H$ such that either $xh_{0}$ is an edge or $h_{0}x$ is an edge. The two cases are symmetric, so we will assume that $xh_{0}$ is an edge. Since $H$ is in-homogeneous, we get that for every $h \in H$, $xh$ is an edge, and thus the out-degree of $x$ is at least $|H|$. Note that since the graph is, by assumption, vertex-transitive, we have that the in-degree and out-degrees are constant, equal to each other, and equal at every vertex.

Now, since $\Gamma$ is anti-connected, it follows that there exists $x' \notin H$, such that $x'$ is adjacent to some $h_{1} \in H$ in $S(\Gamma^{*})$. This means that either $h_{1}x'$ or $x'h_{1}$ is not an edge in $\Gamma$. We assume without loss of generality that $x'h_{1}$ is not an edge. Then, for every $h \in H$, $x'h$ is not an edge and therefore the in-degree of $x'$ is at most $|V(\Gamma )|-|H|$. Thus we conclude that $|H|\leq \frac{|V(\Gamma )|}{2}$. The rest of the proof follows the same way as that of \ref{homogeneous product equivalence}.
\end{proof}

Using the same techniques, we can establish the following result:

\begin{thm}
  Let $G$ be a group and $S \subseteq G\setminus \{1\}$, such that $S \cup S^{-1}$ is distinct from $G\setminus \{1\}$ and $S\neq \emptyset$. Then $\text{Cay}(G,S)$ has a non-trivial bihomogeneous set if and only if there exists a subgroup $H$ of $G$ such that for every $g \in S \setminus H$, $HgH \subseteq S$.
\end{thm}

\section{Homogeneous sets in Cayley graphs of a ring} 
\label{sec:ring_cayley}

Let $R$ be a commutative ring, and let $S$ be a subgroup of the set $R^{\times}$ of units of $R$. Throughout this section, we write $\text{Cay}(R, S)$ for the Cayley graph $\text{Cay}((R,+), S)$. In this section, we delve into the question of determining when $\text{Cay}(R,S)$ is a prime graph.  We will focus on undirected graphs and hence we will assume that $-1 \in S$ for the rest of this section (since $S$ is a subgroup of $R^{\times}$, this is equivalent to the condition that $S=-S$). We remark that this type of Cayley graph is particularly interesting because it captures the interplay between the additive and multiplicative structures of $R$. Here ideals of rings play a key role. We call an ideal $I$ of $R$ non-trivial if $I \neq 0, R.$ Special cases of these Cayley graphs have been extensively studied in the literature. For instance, when $S = R^{\times}$, the resulting graph $\text{Cay}(R, R^{\times})$ is called a unitary Cayley graph (see \cite{bavsic2015polynomials, klotz2007some}). Additionally, when $R$ is the ring of integers modulo a number $N$ and $S$ is associated with a Dirichlet character, the resulting graph is a generalization of the classical Paley graphs (see \cite{cramer2016isoperimetric, paley1933orthogonal, paleygraph}). It is worth noting that due to their arithmetic nature, these Paley graphs exhibit intriguing properties and have found applications in diverse fields such as coding and cryptography theory (see \cite{ghinelli2011codes, javelle2014cryptographie}).

\subsection{Some general properties of $\text{Cay}(R,S)$}

In this section, we prove some foundational results about the primality of $\text{Cay}(R,S)$. We start our investigation with the following theorem.

\begin{thm} \label{thm:prime_implies_ideals}
    Let $R$ be a commutative ring and let $S \subseteq R^{\times}$ be a subgroup. Suppose that ${\rm Cay}(R, S)$ is connected and anti-connected. If ${\rm Cay}(R,S)$ is not prime, then there exists a non-trivial ideal $I$ such that $I$ is a homogeneous set.
\end{thm}

\begin{proof}
    By \cref{thm:subgroup_homogeneous}, if $H$ is a maximal non-trivial homogeneous set of $\text{Cay}(R,S)$ containing $0$, then $H$ is a subgroup of $(R, +).$ We claim that $H$ is an ideal in $R$ as well. First, we observe that if $s \in S \subseteq R^{\times}$, then the multiplication by $s$ is an automorphism of $\text{Cay}(R,S).$ Consequently, $sH$ is also a homogeneous set. Since $0 \in H \cap sH$, we conclude that $H \cup sH$ is also a homogeneous set. By the maximality of $H$, either $sH=H$ or $H \cup sH = R.$ However, the second case cannot happen since
 \[ | H \cup sH| = |H| +|sH| - |H \cap sH| \leq 2 \frac{|R|}{2}-1 <|R|.\]
 We conclude that $sH=H.$ We now claim if $r \in R$, then $rH=H$. In fact, since $\text{Cay}(R,S)$ is connected, we can write
 \[ r = \sum_{i=1}^d m_i s_i,\]
 where $m_i \in \Z$ and $s_i \in S.$ For each $h \in H$, we have 
 \[ rh = \sum_{i=1}^d m_i (s_i h).\]
 Since $s_i h \in H$ and $H$ is a subgroup of $(R, +)$, we conclude that $rh \in H.$ So, we conclude that $H$ is an ideal in $R$. 
\end{proof}

 The following is an immediate corollary of having a homogeneous set which is at the same time an ideal in $R.$ (The fact that the induced subgraph corresponding to $I$ contains no edges follows from the fact that $I$ is an ideal and every element of $S$ is a unit.) For each natural number $n$, we denote by $E_n$ the cocomplete graph on $n$ vertices. 
\begin{cor}
    Let $R$ be a commutative ring and $S \subseteq R^{\times}.$  Let $I$ be a non-trivial ideal in $R$ which is also a homogeneous set in ${\rm Cay}(R,S).$ Let $\Phi\colon R\to R/I$ be the canonical ring map and $n= |I|$. Then
    \[ {\rm Cay}(R,S) \cong {\rm Cay}(R/I, \Phi(S)) \cdot E_n. \]
\end{cor}

For two regular graphs $G,H$ of degree $r_{G}$ and $r_H$ respectively, the spectrum of the wreath product of $G \cdot H = $ is well-known. In fact, by \cref{rem:wreath_join}, we know that $G \cdot H$ is the joined union $G[H, \ldots, H]$.  By \cite[Theorem 3]{joined_union2} and \cite[Theorem 3.3]{CM1_b},  we know that the spectrum of $G \cdot H$  is the union of 
\[ r_H + |H| \Spec(G), \]
and $|G|$ copies of 
\[ \Spec(H) \setminus \{r_H \}. \]

For the case $\text{Cay}(R,S)= {\rm Cay}(R/I, \Phi(S)) \cdot E_n$, we have $r_H =0, |H| = |I|, |G| = |R/I|.$ Therefore, we see that its spectrum is the union of
\[ |I| \Spec(\text{Cay}(R/I, \Phi(S))), \]
and $|R/I|(|I|-1)$ copies of $0$.  

\begin{cor}
    Let $R$ be a commutative ring and $S \subseteq R^{\times}.$ Suppose that ${\rm Cay}(R,S)$ is connected and anti-connected. If ${\rm Cay}(R,S)$ is not prime, then $0$ is an eigenvalue of ${\rm Cay}(R,S)$ with multiplicity at least $\frac{|R|}{2}.$
\end{cor}
\begin{proof}
This follows directly from the fact that 
\[ |R/I|(|I|-1) = |R| - \frac{|R|}{|I|} \geq |R| - \frac{|R|}{2} = \frac{|R|}{2}.
\qedhere\]
\end{proof} 

In the following proposition, we provide necessary and sufficient conditions for an ideal $I$ to be a homogeneous set in $\text{Cay}(R,S).$
\begin{prop} \label{prop:iff}
Let $R$ be a commutative ring and $S \subseteq R^{\times}.$  Let $I$ be an ideal in $R$. Then $I$ is a homogeneous set in ${\rm Cay}(R,S)$ if and only if one of the following conditions holds.
\begin{enumerate}
    \item $I=R.$
    \item $I \neq R$ and 
\[ S+I = \{s+m| s \in S, m \in I \} = S. \]

\end{enumerate}

\end{prop}

\begin{proof}
    First, suppose that $I$ is a homogeneous set. The case $I=R$ is trivial so we assume that $I \neq R.$ Since $S \subseteq R^{\times}$ and $I$ is an ideal, we conclude that $S \cap I = \emptyset.$ Let $x \in S$ and $m \in I.$ Since $(x,0) \in E(\text{Cay}(R,S))$ and $I$ is a homogeneous set, we have $(x,-m) \in E(\text{Cay}(R,S))$ as well. By definition, $x+m \in S.$

    Conversely, let us prove that if the following condition is satisfied
    \[ S+I = \{s+m| s \in S, m \in I \} \subseteq S, \]
    then $I$ is homogeneous. We remark that since $I \cap S = \emptyset$, we have the following partition of $R$
    \[ R = I \bigsqcup S \bigsqcup (R \setminus (I \cup S)) \]
    Let $x \in R \setminus I$. Then we know either $x \in S$ or $x \in R \setminus (I \cup S).$ If $x \in S$ then $(x,m) \in E(\text{Cay}(R,S))$ for all $m \in I.$ Let us consider the case $x \not \in (I \cup S).$ We claim that $(x,m) \not \in E(\text{Cay}(R,S))$ for all $m\in I.$ Suppose to the contrary, then $x-m \in S$ for some $m \in I.$ We then have
    \[ x = (x-m)+m \in S+I = S.\]
    This shows that $x \in S$, which is a contradiction. 
\end{proof}

\begin{cor}
    Suppose that $I, J$ are two ideals in a commutative ring $R$ such that both of them are homogeneous sets in ${\rm Cay}(R,S)$. Then $I+J$ is also a homogeneous set.
\end{cor}

We now recall the definition of the Jacobson radical of a commutative ring. 
\begin{defn} \label{defn:jacobson}
Let $R$ be a commutative ring. The Jacobson radical of $R$ is defined the the following equivalent definitions (see \cite[Section IV]{ mcdonald1974finite} and \cite[Chapter 4.3]{pierce1982associative}).

\begin{enumerate}
    \item $\Rad(R) = \cap M $ where $M$ runs through the set of all maximal ideals of $R.$
    \item $\Rad(R) = \left\{r \in R| 1- rs \in R^{\times} \text{ for all } s \in R \right \}$.
    \item $\Rad(R)$ is an ideal and it is the largest ideal $K$ such that $1-r$ is a unit in $R$ for all $r \in K.$
    \item $\Rad(R)$ is the largest nilpotent ideal in $R.$
\end{enumerate}
\end{defn}

Now, suppose that $I$ is a non-trivial ideal in $R$ such that $I$ is a homogeneous set in ${\rm Cay}(R, S)$. By \cref{prop:iff} we know that 
\[ 1+ I = \left \{1 + m| m  \in I \right \} \subset S \subset R^{\times}.\]
Therefore, by Condition $3$ in \cref{defn:jacobson}, we conclude that $I \subset \Rad(R).$
Thus we have proved the following.
\begin{prop} \label{prop:contained_jacobson}
    Each non-trivial homogeneous ideal in $R$ is a subset of $\Rad(R).$ In particular, each such ideal is nilpotent. 
\end{prop}
We recall that a finite commutative ring is called semisimple if $\Rad(R)=R$ (see \cite[Proposition a, Chapter 3]{pierce1982associative} and take into account that finite commutative ring is Artinian). By \cref{thm:prime_implies_ideals} and \cref{prop:contained_jacobson} we have the following corollary.

\begin{cor}
     Let $R$ be a semisimple commutative ring and let $S \subseteq R^{\times}$ be a subgroup. Suppose that ${\rm Cay}(R, S)$ is connected and anti-connected. Then ${\rm Cay}(R,S)$ is prime. 
\end{cor}
We now study the functorial properties of $\text{Cay}(R,S)$. For this, we introduce the following definition.
\begin{defn} \label{def:homomorphism}
Let $R,R'$ be rings and let $S,S'$ be subgroups of $R^{\times}, (R')^{\times}$ respectively. 

\begin{enumerate}
    \item A homomorphism from $(R,S) \to (R',S')$ is a ring homomorphism $\Phi: R \to R'$ such that $S=\Phi^{-1}(S').$ 
    \item Without causing any confusion, we will denote this homomorphism by $\Phi$ as well. 
 We say that $\Phi$ is surjective if the ring map $\Phi: R \to R'$ is.    
 \end{enumerate}
\end{defn}
With this definition, we have the following proposition. 

\begin{prop} \label{prop:47}
    Let $R,R'$ be rings and let $S,S'$ be subgroups of $R^{\times}, (R')^{\times}$ respectively. Let $\Phi\colon  (R, S) \to (R',S')$ be a surjective homomorphism. Let $I' \neq R'$ be an ideal in $R$ such that $I'$ is a homogeneous set in ${\rm Cay}(R', S').$ Then $I=\Phi^{-1}(I')$ is an ideal in $R$ which is also a homogeneous set in ${\rm Cay}(R,S).$
\end{prop}

\begin{proof}
    By definition $\Phi^{-1}(I)$ is an ideal in $R.$ We note that since $S' \cap I' = \emptyset$, we must have $I \neq R$. To show that $I$ is homogeneous, we only need to show that 
    \[ S + I = S.\]
    In fact, we have 
    \[ \Phi(S+I) = \Phi(S)+\Phi(I) = S' + I' = S'. \]
    Therefore 
    \[ S+I \subseteq \Phi^{-1}(S') = S. \]
    Since we always have $S \subseteq S+I$, we conclude that $S+I = S.$
\end{proof}

\begin{rem}
    We compare \cref{def:homomorphism} to the notion of a graph homomorphism defined in \cref{def:graph_morphism}. We remark that with \cref{def:graph_morphism}, if $f(x) = f(y)$, then $x$ and $y$ are non-adjacent in $G$. We  also note that in general, if $f$ is a graph homomorphism, then $f^{-1}(X)$, where $X$ is a homogeneous set in $H$, need not be a homogeneous set in $G$. For this to be true, we need the modified assumption that if $f(x) \neq f(y)$, then $xy \in E(G)$ if and only if $f(x)f(y) \in E(H)$ (note that this assumption does not imply that $f$ is a graph homomorphism, as it allows for adjacent vertices to have the same image). The setting of \cref{prop:47} in fact gives us a graph homomorphism from $\text{Cay}(R, S)$ to $\text{Cay}(R', S')$ which also satisfies the modified assumption. 
\end{rem}


\begin{lem}
Let $R$ be a commutative ring and let $S \subseteq R^{\times}.$  Let $I$ be an ideal in $R$ which is also a homogeneous set in ${\rm Cay}(R,S)$ such that $I \neq R.$ Let $\Phi\colon R \to R/I$ be the canonical ring map. Then $\Phi^{-1} (\Phi(S))=S.$ In other words, the induced map $\Phi\colon (R,S) \to (R/I, S/I)$ is a surjective homomorphism.
\end{lem}
\begin{proof} 
    We always have $S \subseteq\Phi^{-1}(\Phi(S)).$ Let $t \in \Phi^{-1}(\Phi(S)).$ Then $\Phi(t) \in \Phi(S).$ That means we can find $s \in S$ such that $\Phi(t)=\Phi(s)$. Equivalently, $t-s =m$ for some $m \in I.$ We then have 
    \[ t = s+m. \]
    However, since $I$ is a homogeneous set, from \cref{prop:iff} we know that $I+S =S.$ This shows that $t \in S.$ We conclude that $\Phi^{-1}(\Phi(S)) \subseteq S$ and hence $\Phi^{-1}(\Phi(S)) = S.$
\end{proof}

We recall that a positive integer $p$ is a \emph{prime number} if it has no non-trivial factor. Said differently, $p$ is a prime number if for every $n \geq 1$ and every surjective homomorphism $ \Phi:\Z/p \to \Z/n$, we have that $\Phi$ is an isomorphism. Inspired by this fact, we make the following definition. 
\begin{defn}
    Let $R$ be a ring and $S \subseteq R^{\times}.$ We say that the pair $(R,S)$ is primitive if it is minimal in the following sense: if $\Phi:(R,S) \to (R',S')$ is a surjective homomorphism, then it is an isomorphism. 
\end{defn}

Roughly speaking, the pair $(R,S)$ is primitive if it has no non-trivial quotients. Combining the results that we proved so far, we have: 

\begin{prop}
Let $R$ be a commutative ring and $S \subseteq R^{\times}$. Suppose that ${\rm Cay}(R, S)$ is connected and anti-connected. Then ${\rm Cay}(R,S)$ is prime if and only if the pair $(R,S)$ is primitive.
\end{prop}

\begin{proof}
    If $\text{Cay}(R,S)$ is not prime, then there exists a non-trivial ideal $I$ such that $I$ is a homogeneous set. We then have a surjective homomorphism $\Phi\colon (R,S) \to (R/I, S/I)$. Since $I$ is non-trivial, it follows that $\Phi$ is not an isomorphism. 

    Conversely, suppose that $\text{Cay}(R,S)$ is a prime graph. Let $\Phi\colon (R,S) \to (R',S')$ be a surjective homomorphism. Since $\{0 \}$ is a homogeneous ideal in $\text{Cay}(R',S')$, we know that $\ker(\Phi) = \Phi^{-1}(0)$ is a homogeneous set in $\text{Cay}(R,S)$ as well. Since $\text{Cay}(R,S)$ is prime, we conclude that $\ker(\Phi)= \{0\}.$ This shows that $\Phi$ is an isomorphism.
\end{proof}

\subsection{Generalized Paley graphs}
In this section, we will focus on the case where $S$ is associated with a multiplicative function on $R.$ To start, let us consider $\psi\colon R \to \C$ to be a multiplicative function, namely 
\[ \psi(ab) = \psi(a) \psi(b), \forall a, b \in R, \]
$\psi(0)=0$ and $\psi(1)=1$.     We say that $\psi$ is an even function if $\psi(-1)=1.$

\begin{defn}
    The Paley digraph $P_{\psi}$ is defined as the Cayley graph $\text{Cay}((R,+), \ker(\psi))$ where
    \[ \ker(\psi) = \{a \in R| \psi(a)=1 \}. \]
\end{defn}

 We can see easily that $\text{Cay}(R, \ker(\psi))$ is an undirected graph if and only if $\psi$ is even. From now on, we will always assume that $\psi$ is even.

The name generalized Paley digraph is motivated by the classical Paley digraph ${\rm Cay}(\F_q, \ker(\psi))$ where $\F_q$ is a finite field with $q$ elements and $\psi: \F_q \to \C$ is defined by

\[ \psi(a)  = \begin{cases}
   0 & \text{if $a=0$ } \\ 
  1  &  \text{if $a \in \F_{q}^{\times}$ and $a$ is a square in $\F_q$ } \\
  -1 &  \text{else. }
\end{cases} \]  

These digraphs are named after the English mathematician  Raymond Paley (see \cite{paleygraph}).
\begin{rem}
    By definition, $1 \in \ker(\psi).$ Consequently, $P_{\psi}$ is not an empty graph. Said differently, its complement $P_{\psi}^c$ is not a complete graph. 
\end{rem}

\begin{rem}
    We remark that in the above definition, we do not require that $\ker(\psi)$ is a subgroup of $R^{\times}.$ However, as we will see later, this property is automatic if $\psi$ is primitive, a concept that we will define below (see Lemma \ref{lem:primitive_implies_unit}).
\end{rem}

\begin{defn}
Let $I$ be a non-trivial ideal in $R$. Let $\psi\colon R/I \to \C$ be a multiplicative function. Let $\widehat{\psi_I}\colon R \to R/I \to \C$ be the induced multiplicative function. Such $\widehat{\psi_I}$ is called a non-primitive multiplicative function, and we say that $\widehat{\psi_I}$ factors through $R/I$. If $\psi: R \to \C$ is a multiplicative function which is not of the form $\widehat{\psi_I}$ for non-trivial $I$, then we say that $\psi$ is a primitive.
\end{defn} 
From now on, without further notice, we will assume that $R$ is a commutative ring. 

\begin{prop}
\label{prop:not_primitive_implies_not_prime}
   If $\psi$ is not primitive, then $P_{\psi}$ is not prime. 
\end{prop}
\begin{proof}
Since $\psi$ is not primitive, we can find a non-trivial ideal $I$ such that $\psi = \widehat{\psi_I}$ where $\psi_I: R/I \to \C$ is a multiplicative function. We claim that $I$ is a homogeneous set in $P_{\psi} = P_{\widehat{\psi_I}}$. In fact, let $a \in R$ and $ a \not \in I.$ Then for $h \in I$, we have $\psi(a)=\psi(a-h)=\psi(h-a),$ using that $\psi$ is even. Consequently, if $\psi(a)=1$ then $(a,h) \in E(P_{\psi})$ for all $h \in I.$ Otherwise, $(a,h) \not \in E(P_{\psi})$ for all $h \in I.$ By definition, $I$ is a homogeneous set and hence $P_{\psi}$ is not prime.
\end{proof} 

It turns out that the converse of this proposition also holds under some mild assumptions. To show this fact, we first introduce the following lemma. 


\begin{lem} \label{lem:primitive_implies_unit}
        Let $R$ be a finite commutative ring. Suppose that $\psi: R \to \C$ is primitive. Then 
        \[ R^{\times} = \{a | \psi(a) \neq 0 .\} .\]
        In particular, $\ker(\psi) \subseteq R^{\times}.$
\end{lem}
\begin{proof}
    If $a \in R^{\times}$ then 
    \[ 1 = \psi(1) = \psi(a) \psi(a^{-1}). \] 
    This shows that $\psi(a) \neq 0$. Conversely, suppose $a \in R$ such that $\psi(a) \neq 0.$ We will show that $a \in R^{\times}.$ The multiplication map $m_a : R \to R$ sending $x \mapsto ax$ is a group homomorphism. Let $I=\ker(m_a).$ We claim that $\psi(x+b)=\psi(x)$ for all $x \in R$ and $b \in I.$ In fact, we have 
    \[ \psi(a) \psi(x) = \psi(ax)=\psi(a(x+b))=\psi(a) \psi(x+b).\]
    Since $\psi(a) \neq 0$, we conclude that $\psi(x)=\psi(x+b).$ This shows that $\psi$ factors through $R/I.$ Since $\psi$ is primitive, we conclude that $I= \{0 \}$ and consequently $a \in R^{\times}.$
\end{proof}


\begin{lem}
    Assume that $\psi: R \to \C$ is primitive. Let $I$ be an ideal of $R$ such that $I$ is a homogeneous set in $P_{\psi}.$ Then $I= \{0 \}$ or $I=R.$
\end{lem}
\begin{proof}
    Let us assume that $I \neq R$. We claim that $\psi$ factors through $R \to R/I \to \C.$ More precisely, we will show that if $x \in R$ and $m \in I$ then 
    \[ \psi(x)=\psi(x+m).\]
    If $\psi(x)=\psi(x+m)=0$, then we are done. Otherwise, either $\psi(x) \neq 0$ or $\psi(x+m) \neq 0$. First, let us assume that $\psi(x) \neq 0$. Then $x \in R^{\times}$ by \cref{lem:primitive_implies_unit}. By definition, $\psi(1)=1$ and therefore $(0,1) \in E(P_{\psi})$. Since $I \neq R$, it follows that $1 \not \in I$. Furthermore, $I$ is an ideal which is also a homogeneous set, we conclude that $(-x^{-1}m, 1) \in E(P_{\psi})$ as well. Consequently $\psi(x^{-1}m+1)=1.$ This shows that $\psi(x)=\psi(x+m)$. 

    For the case $\psi(x+m) \neq 0$, by an identical argument applied to the pair $(y, y-m)$ with $y=x+m$, we also have $\psi(x)= \psi(x+m).$ We conclude that in all cases,  $\psi(x)=\psi(x+m)$. Since $\psi$ is primitive, it follows that $I= \{0 \},$ as desired. 
\end{proof}


\begin{cor} \label{cor:non_hom_set}
    If $\psi: R \to \C$ is primitive, then $P_{\psi}$ has no non-trivial homogeneous set which is also an ideal in $R.$
\end{cor}


\begin{lem} \label{lem:anti_connnected}
    Suppose $\psi$ is primitive. Assume further that $P_{\psi}$ is connected but $P_{\psi}$ is not a complete graph. Then $P_{\psi}$ is anti-connected.
\end{lem}

\begin{proof}
    Let $S= \ker(\psi)$ and $S^c= R \setminus S$. Then the complement of $P_{\psi}$ is $\Gamma(R, S^c \setminus \{0 \}).$ This complement is connected if and only if $S^c$ generates $(R,+)$ as an abelian group. We first claim that the abelian group $I = \langle S^c \rangle$ generated by $S^c$ is an ideal in $R.$ In fact, let $s_c \in S^c$ and $s \in S.$ Then 
    \[ \psi(ss_c)= \psi(s)\psi(s_c) = \psi(s_c) \neq 1.\]
    This shows that $ss_c \in S^c.$ Now, let $r \in R$ be an arbitrary element. Since $S$ generates $R$ as an abelian group, we can write 
    $ r = \sum_{i=1}^t m_i s_i,$ where $s_i \in S$ and $m_i \in \Z.$ We then have 
    \[ rs_c = \sum_{i=1}^t m_i (s_i s_c) \in I. \]
    Suppose that $y \in I.$ By definition, we can write $y = \sum_{i=1}^{p} n_i (s_c)_i$  where $n_i \in \Z$ and $(s_c)_i \in S^c.$ We then have 
    \[ ry = \sum_{i=1}^p n_i (r (s_c)_i) \in I.\]
    Since this is true for all $r \in R$ and $y \in I$, we conclude that $I$ is an ideal in $R.$
    
    We claim that $I=R.$ Suppose to the contrary that $I \neq R$. Let $m$ be an arbitrary maximal ideal in $R$.  If $x \in m$ then $\psi(x)=0$ by \cref{lem:primitive_implies_unit}. Consequently, $x \in S^c$ and hence $x \in I.$ Since this is true for every $x$, we conclude that $m \subseteq I.$ We conclude that $I$ is the only maximal ideal in $R$. In other words, $R$ is a local ring and its maximal ideal is $I.$ By \cref{lem:primitive_implies_unit}, we must have $\psi(I)=0$. We note that since $P_{\psi}$ is not complete, $S \neq R \setminus \{0 \}$. Therefore,  $S^c$ contains a non-zero element. Consequently, $I \neq 0$. Additionally, because $\psi$ is primitive, we can find $r \in R \setminus I $ such that $\psi(r) \neq 1$ (otherwise, we have  $\psi(I)=0$ and $\psi(x)=1$ for each $x \in S = R \setminus I.$ In other words, $\psi$ factors through the nontrivial quotient $R \to R/I.$) This would be impossible since we then have $r \in S^c \subseteq I.$
\end{proof}

Combining \cref{cor:non_hom_set}, \cref{lem:anti_connnected} and \cref{thm:prime_implies_ideals}, we have the following proposition which is the converse of Proposition \ref{prop:not_primitive_implies_not_prime}.

\begin{thm} \label{prop:criterior}
    Let $\psi: R \to \C$ be a primitive multiplicative function. Assume further that $P_{\psi}$ is connected and $P_{\psi}$ is not complete. Then $P_{\psi}$ is prime.
\end{thm}

\begin{rem} \label{rem:counter_example}
    The condition that $P_{\psi}$ is connected is necessary for \cref{prop:criterior}. For example, let $R=\Z/2 \times \Z/2$ and $\psi: R \to \C$ be the function such that $\psi((a,b))=0$ if $(a,b) \neq (1,1)$ and $\psi((1,1))=1.$ Then, $\psi$ is primitive but $P_{\psi}$ is not connected. In fact, $X_1 = \{(0,0), (1,1) \}$ and $X_2 = \{(1,0), (0,1)\}$ are two connected components of $R.$ We refer the reader to \cref{sec:unitary} for some further discussions for this type of graph.
\end{rem}

\begin{rem}
Observe that if $\psi: R \to \C$ is a primitive multiplicative function such that $P_{\psi}$ is complete then $R$ is a field (note that a complete graph $K_n$ is not prime if $n \geq 3$). Indeed, if $P_{\psi}$ is complete then $\ker(\psi) = R \setminus \{0\}.$ By \cref{lem:primitive_implies_unit}, we conclude that $R^{\times} = R \setminus \{0 \}.$ This implies that $R$ is a field. 

\end{rem}

\subsection{Unitary Cayley graphs} \label{sec:unitary}
We continue with some discussions on the unitary graph of a finite commutative ring. We first introduce this concept (see \cite{unitary, klotz2007some} for some further studies regarding it.) 

\begin{defn}
Let $\one: R \to \mathbb{C}$ be the principal multiplicative function defined by 
   \[ \one(r) = \begin{cases}
  1  & \text{if } r \in R^{\times}, \\
  0  & \text{else. } \end{cases} \] 
The unitary graph associated with $R$ is $\text{Cay}(R, \ker(\one))=\text{Cay}(R, R^{\times})$. We will denote this graph by $X_R.$
\end{defn}
\begin{question}
    When is $X_R$  prime? 
\end{question}
To answer this question, we recall the following definition of the tensor product of two graphs. In the literature, other names are used for this concept including direct product and Kronecker product (see \cite[Page 36]{hammack2011handbook}).
\begin{defn}
    Let $G,H$ be two graphs. The tensor product $G\times H$ of $G$ and $H$ is the graph with the following data:
    \begin{enumerate}
    \item The vertex set of $G \times H$ is the Cartesian product $V(G) \times V(H)$; and
    \item Two vertices $(g,h)$ and $(g',h')$ are adjacent in $G \times H$ if and only if $(g,g') \in E(G)$ and $(h, h') \in E(H).$
\end{enumerate}
\end{defn}
We remark that if $R=R_1 \times R_2$, then an element $r=(r_1, r_2) \in R$ is a unit if and only if $r_1, r_2$ are both units. From this observation, we see that $X_{R} \cong X_{R_1} \times X_{R_2}$.

Let $R^{\s}$ be the semisimplification of $R$, namely 
\[ R^{\s} = R/Rad(R) .\] 
where $\Rad(R)$ is the Jacobson radical of $R$ as defined in \cref{defn:jacobson}. There is a canonical ring map 
\[ \Phi: R \to R^{\s} .\] 

\begin{prop} \label{prop:radical}
The map $\Phi$ has the following properties. 
\begin{enumerate}
\item $\Phi$ is surjective. 
\item Let $r \in R$. Then $r$ is a unit in $R$ if and only if $\Phi(r)$ is a unit in $R^{\s}.$
\end{enumerate}
\end{prop}

\begin{proof} Clearly $\Phi$ surjective. Now let $x$ be an element in $R$. Suppose that $\bar{x}=\Phi(x)$
is a unit in $R^{\s}.$ By definition, there exists $y \in R$ such that $xy-1 \in \Rad(R).$ By the definition of the Jacobson radical, this implies that 
\[ xy=1+(xy-1) ,\] 
is a unit in $R.$ Consequently, $x$ is a unit in $R.$ The converse statement is clear since $\Phi$ is a ring homomorphism.
\end{proof}

By Proposition \ref{prop:radical}, we see that $\one$ factors through $R \to R/ \Rad(R).$ Therefore, $\one$ is not primitive and hence by \cref{prop:not_primitive_implies_not_prime}, $X_R$ is not prime unless $R=R^{\s}$, namely, $R$ is semisimple. From now on, we will assume that $R$ is semisimple. By the structure theorem, it follows that $R$ is a product of fields $R = \prod_{i=1}^d k_i$ where $k_i$ are fields of size $q_i$. We have $R^{\times} = \prod_{i=1}^d k_i^{\times}.$ From this decomposition, we see that $X_{R} \cong  \prod_{i=1}^d X_{k_i} = \prod K_{q_i}$ where $K_n$ is a complete graph on $n$ nodes. Therefore, our question can be reduced to the following question. 
\begin{question}
    Let $n_1, \ldots, n_d$ be positive integers. When is $\prod_{i=1}^d K_{n_i}$ prime? 
\end{question}

We remark that the vertex set of this graph is the set of $d$-tuple $(s_1, s_2, \ldots, s_d)$ where $0 \leq s_i \leq n_i-1.$ Two vertices $(s_1, s_2, \ldots, s_d)$ and $(t_1, t_2, \ldots, t_d)$ are connected if and only if $s_i \neq t_i$ for all $1 \leq i \leq d.$  We remark that the complement of $\prod_{i=1}^{d} K_{n_i}$ is always  connected if $d \geq 2$. In other words, $\prod_{i=1}^d K_{n_i}$ is anti-connected if $d \geq 2.$

\begin{rem}
If $d=1$ then $X_{R} \cong K_{q_1}$ and we can see that $K_{q_i}$ is not prime unless $q_i=2$. 
\end{rem}
From now on, we will assume that $d>1.$ From \cref{rem:counter_example}, we know that $K_2 \times K_2$ is not connected. Additionally, we know that the tensor product behaves well with respect to the disjoint union of graphs (see \cite[Page 54]{hammack2011handbook}). Specifically, if $G, H_1, H_2$ are graphs then 
\[ G \times (H_1 + H_2) \cong (G \times H_1) + (G \times H_2) .\]
Here $H_1+H_2$ is the disjoint union of $H_1$ and $H_2.$ From this, we can see that if there is more than one $n_i=2$, the product $\prod_{i=1}^{d} K_{n_i}$ is not connected.

\begin{lem} \label{lem:connectedness}
Suppose that $d \geq 2.$    The graph $\prod_{i=1}^d K_{n_i}$ is connected if and only if there is at most one $n_i=2.$
\end{lem}

\begin{proof}
    The if part has been proved above. Let us prove the ``only if'' part. Let $(s_1, s_2, \ldots, s_d)$ and $(t_1, t_2, \ldots, t_d)$ be elements of $\prod_{i=1}^d K_{n_i}.$ If $n_i>2$ for all $i$, we can find $(u_1, u_2, \ldots, u_d)$ such that 
    \[ u_i \not \in \{s_i, t_i \}, \forall 1 \leq i \leq d. \]
    By definition, $(u_1, u_2, \ldots, u_d)$ is connected to both $(s_1, s_2, \ldots, s_d)$ and $(t_1, t_2, \ldots, t_d)$. Thus, we find a path 
    \[ (s_1, s_2, \ldots, s_d) \to (u_1, u_2, \ldots, u_d) \to (t_1, t_2, \ldots, t_d). \]
    Now, suppose that one of the $n_i$ is $2.$ Without loss of generality, we can assume that $n_1=2.$ Again, we can find $(u_1, u_2, \ldots, u_d)$ such that 
        \[ u_i \not \in \{s_i, t_i \}, \forall 2 \leq i \leq d. \]
    If $s_1 = t_1$, we can find $u_1$ such that $u_1 \neq t_1$. Then we have the following path: 
        \[ (s_1, s_2, \ldots, s_d) \to (u_1, u_2, \ldots, u_d) \to (t_1, t_2, \ldots, t_d). \]
    If $s_1 \neq t_1$, we have the following path 
            \[ (s_1, s_2, \ldots, s_d) \to (t_1, u_2, \ldots, u_d). \]
    By the previous case, there is a path from $(t_1, u_2, \ldots, u_d) \to (t_1, t_2, \ldots, t_d).$ This shows that there is a path from      $(s_1, s_2, \ldots, s_d)$ to $(t_1, t_2, \ldots, t_d).$ 
\end{proof}

\begin{rem}
\cref{lem:connectedness} can also be deduced from \cite[Corollary 5.10]{hammack2011handbook} which asserts that a tensor product of connected nontrivial graphs is connected if and only at most one of the factors is bipartite. In our case, $K_{n_i}$ is connected. It is bipartite if and only if $n_i=2.$
    
\end{rem}

Now back to our problem of whether $X_{R}$ is prime where $R = \prod_{i=1}^{d} k_i.$ As we explained above, we have 
\[ X_{R} \cong \prod_{i=1}^d X_{k_i} \cong \prod_{i=1}^d K_{q_i}, \]
where $q_i=|k_i|.$ Clearly, the multiplicative function $\one: R \to \C$ is primitive. We also know that $X_{R}$ is anti-connected. Therefore, by Proposition \ref{prop:criterior}, we conclude $X_{R}$ is prime if and only if it is connected. By the above lemma, this happens if and only if at most one $q_i=2.$ In summary, we have proved the following.
\begin{thm}
    For a commutative ring $R$, the graph $X_{R}$ is prime if and only if $R =\F_2$ or $R \cong \prod_{i=1}^d k_i$ where $d \geq 2$ and $k_i$ are fields such that there is at most one $i$ such that $|k_i|=2.$
\end{thm}

\bibliographystyle{amsplain}
\bibliography{references.bib}
\end{document}